\documentclass[10pt,reqno]{amsart}
\usepackage[T2A]{fontenc}
\usepackage{amsmath}
\usepackage{amssymb}
\usepackage{amsfonts}
\usepackage{stmaryrd} 
\usepackage{multirow}

\newtheorem{thm}{Theorem}
\newtheorem{prop}[thm]{Proposition}

\newtheorem{prob}{Problem}

\begin{document}

\title[Generation by conjugate elements]
{Generation by conjugate elements of finite almost simple groups with a sporadic socle}

\author{Danila O. Revin}%
\address{Danila O. Revin
\newline\indent Siberian Federal University,
\newline\indent 79, Svobodny av.
\newline\indent 660041, Krasnoyarsk, Russia
\newline\indent ORCID:\,0000-0002-8601-0706
} \email{revin@math.nsc.ru}

\author{Andrei V. Zavarnitsine}%
\address{Andrei V. Zavarnitsine
\newline\indent Sobolev Institute of Mathematics,
\newline\indent 4, Koptyug av.
\newline\indent 630090, Novosibirsk, Russia
\newline\indent ORCID:\,0000-0003-1983-3304
} \email{zav@math.nsc.ru}

\maketitle {\small
\begin{quote}
\noindent{\sc Abstract. } As defined by Guralnick and Saxl given a nonabelian simple group $S$ and its nonidentity automorphism $x$, a natural number $\alpha^{\phantom{S}}_{S}(x)$ does not exceed a natural number $m$ if some $m$ conjugates of $x$ in the group $\langle x,S\rangle$ generate a subgroup that includes~$S$. The outcome of this paper together with one by Di Martino, Pellegrini, and Zalesski, both of which
are based on computer calculations with character tables, is a refinement of the estimates by Guralnick and Saxl on the value of $\alpha^{\phantom{S}}_{S}(x)$ in the case where $S$~is a sporadic group. In particular, we prove that $\alpha^{\phantom{S}}_{S}(x)\leqslant 4$, except when $S$~is one of the Fischer groups and $x$~is a $3$-transposition. In the latter case, $\alpha^{\phantom{S}}_{S}(x)=6$ if $S$~is either $Fi_{22}$ or $Fi_{23}$ and $\alpha^{\phantom{S}}_{S}(x)=5$ if $S={Fi_{24}}'$.
\medskip

 \end{quote}
}

\section{Introduction}

We only consider finite groups in this paper, so the word ``group'' will always mean a finite group.

Let $S$~be a nonabelian simple group, which we always identify with the subgroup  $\operatorname{Inn}(S)$ of inner automorphisms in the group $\operatorname{Aut}(S)$ of all automorphisms. Then $S$~is the unique minimal normal subgroup of every group $G$ such that $S\leqslant G\leqslant \operatorname{Aut}(S)$. In this case, $G$ is usually called an \emph{almost simple group with socle}~$S$. If $x\in \operatorname{Aut}(S)$~is a nonidentity (possibly, inner) automorphism then  the subgroup of $G=\langle x,S\rangle$ that is generated by the conjugacy class of~$x$ is normal in~$G$. Consequently, this subgroup includes $S$ and, therefore, coincides with~$G$.

In 2003, R.\,Guralnick and J.\,Saxl \cite{GS} introduced the notation   $$\alpha(x)=\alpha^{\phantom{S}}_{S}(x)$$ for the minimum number of elements conjugate to $x\ne1$ in $G=\langle x,S\rangle$ that generate~$G$. In other words, the parameter  $\alpha^{\phantom{S}}_{S}(x)$ is defined by the property that, for every natural $m$, $\alpha^{\phantom{S}}_{S}(x)\leqslant m$ if and only if some elements $x_1,\dots,x_m$ conjugate to~$x$ in~$G$ generate~$G$.  It is not difficult to see that if $y\ne 1$~is a power of~$x$ then, for all   $g_1,\dots , g_m\in G$, we have the inclusion
$$\langle y^{g_1},\dots, y^{g_m}\rangle\leqslant \langle x^{g_1},\dots, x^{g_m}\rangle.$$
Thus, if  $S\leqslant \langle y^{g_1},\dots, y^{g_m}\rangle$ then $S\leqslant \langle x^{g_1},\dots, x^{g_m}\rangle$ and so $\alpha^{\phantom{S}}_{S}(x)\leqslant \alpha^{\phantom{S}}_{S}(y)$. This means in particular that in order to find upper bounds on $\alpha^{\phantom{S}}_{S}(x)$ for a fixed $S$, it is sufficient to consider only elements~$x$ of prime order.

The main result of~\cite{GS} constitutes finding explicit, albeit not always best possible, upper bounds on $\alpha^{\phantom{S}}_{S}(x)$ for all nonabelian simple groups~$S$.
These bounds and their refinements have been extensively used in applications of the classification of finite simple groups. For example, they are substantially used in proofs of various analogues of the famous Baer-Suzuki theorem, see \cite{FGG,GGKP,GGKP1,GGKP2,Gu,GuestLevy,YRV,YWR,YWRV,WGR,RZ}.
For practical use, the estimates on $\alpha^{\phantom{S}}_{S}(x)$ from~\cite{GS} are not always sufficient. Refinements of these estimates for certain simple groups were required and obtained, for example, in \cite{WGR,RZ,DMPZ}.

In the case where $S$~is a sporadic simple group and $x$~is its nonidentity  \emph{inner} automorphism, nearly precise values of $\alpha^{\phantom{S}}_{S}(x)$ were found by
L.\,Di\,Mar\-tino, M.\,A.\,Pel\-legrini, and A.\,E.\,Za\-les\-ski in \cite[Theorem~3.1]{DMPZ}. All cases of imprecise estimates in this theorem are in the following list\footnote{Here and later on we use the notation and conventions from the Atlas of finite groups~\cite{atlas}; in particular, we may use the same symbol to denote both a group element and its conjugacy class.}:

\begin{enumerate}
    \item $(S,x)=(Fi_{22},2A)$ and $5\leqslant \alpha^{\phantom{S}}_{S}(x)\leqslant 6$;\label{item1}
    \item $(S,x)=(Fi_{23},2A)$ and $5\leqslant \alpha^{\phantom{S}}_{S}(x)\leqslant 6$;\label{item2}
    \item $(S,x)=(Fi_{22},3B)$ and $2\leqslant \alpha^{\phantom{S}}_{S}(x)\leqslant 3$;\label{item3}
    \item $(S,x)=(Suz,3A)$ and $3\leqslant \alpha^{\phantom{S}}_{S}(x)\leqslant 4$;
    \item $S=M$, $x$~is not an involution and $2\leqslant \alpha^{\phantom{S}}_{S}(x)\leqslant 3$;
    \item $S=M$, $x$~is an involution and $3\leqslant \alpha^{\phantom{S}}_{S}(x)\leqslant 4$.\label{item6}
\end{enumerate}

As a matter of fact, we can indicate the exact value of $\alpha^{\phantom{S}}_{S}(x)$ in cases \ref{item1}) and \ref{item2}) above. The Fischer groups $Fi_{22}$, $Fi_{23}$, and $Fi_{24}$~are so called {\em $3$-transposition groups}. Each of them is generated\footnote{$Fi_{24}$ is not simple, but its subgroup~${Fi_{24}}'$ of index~$2$ is. Therefore, the generating class of $3$-transpositions for $Fi_{24}$ lies in $Fi_{24}\setminus {Fi_{24}}'$.} by a conjugacy class  $D$ that is a class of \emph{$3$-transpositions}, i.\,e. consists of elements of order $2$ (\emph{involutions}) such that the product of every two of them has order $1$, $2$, or~$3$. The classes of $3$-transpositions in $Fi_{22}$, $Fi_{23}$, and $Fi_{24}={Fi_{24}}'.2$ are $2A$, $2A$, and $2C$, respectively, in the notation of \cite{atlas}. The groups of $3$-transpositions that can be generated by at most five $3$-transpositions were classified in 1995 by J.\,Hall and L.\,Soicher \cite[Theorems (1.1)--(1.3)]{HS}. This result and  S.\,Norton's paper \cite{N} imply that $F_{24}$ can be generated by five $3$-transpositions, whereas $Fi_{22}$ and $Fi_{23}$ cannot. Also, $Fi_{24}$ cannot be generated by four $3$-transpositions, see \cite[Theorem (1.1)]{HS}. Thus, the following assertion holds which in particular gives precise values for $\alpha^{\phantom{S}}_{S}(x)$ in cases \ref{item1}) and \ref{item2}) above.

\begin{prop}\label{Prop}
If $(S,x)\in\{(Fi_{22}, 2A),(Fi_{23}, 2A)\}$ then $\alpha^{\phantom{S}}_{S}(x)=6$. If  $(S,x)=({Fi_{24}}', 2C)$ then $\alpha^{\phantom{S}}_{S}(x)=5$.
\end{prop}

As we have already mentioned, only inner automorphisms of sporadic groups were considered in \cite{DMPZ}. For automorphisms in  $\operatorname{Aut}(S)\setminus S$, where $S$ is a sporadic group, only the estimates from~\cite{GS} are known. If $\operatorname{Aut}(S)\ne S$ and $x\in \operatorname{Aut}(S)\setminus S$~is of prime order then $S$ is as given in the first column of Table~\ref{Spor}, $x$~is an involution whose conjugacy class is given in the third column, and the estimate on $\alpha^{\phantom{S}}_{S}(x)$ from ~\cite{GS} is given in the fourth column. Observe that since two involutions always generate a solvable group, we have $3\leqslant \alpha^{\phantom{S}}_{S}(x)$ for all cases included in Table~\ref{Spor}.

{\small
\begin{center}
\begin{table}
\caption{Sporadic groups $S$ with $\operatorname{Aut}(S)\ne S$}\label{Spor}
\begin{tabular}{|l|c|c|c|}\hline
\multirow{2}{*}{~~$S$} & \multirow{2}{*}{$|S|$} & classes of involutions \rule{0pt}{2.3ex} & \multirow{2}{*}{$\alpha^{\phantom{S}}_{S}(x)\leqslant $}\\
&& $x$ in $\operatorname{Aut}(S)\setminus S$ \rule[-0.9ex]{0pt}{0pt} &\\
\hline\hline
$M_{12}$ & $2^6\cdot 3^3\cdot 5\cdot 11$ \rule{0pt}{2.5ex} & $2C$ &  4  \\
$M_{22}$ & $2^7\cdot 3^2\cdot 5\cdot 7\cdot 11$ & $2B$, $2C$ &  4  \\
$J_{2}$ & $2^7\cdot 3^3\cdot 5^2\cdot 7$ & $2C$ &  4  \\
$J_{3}$ & $2^7\cdot 3^5\cdot 5\cdot 17\cdot 19$ & $2B$ &  4  \\
$McL$ & $2^7\cdot 3^6\cdot 5^3\cdot 7\cdot 11$ & $2B$ &  4  \\
$O'N$ & $2^9\cdot 3^4\cdot 5\cdot 7^3\cdot 11\cdot 19\cdot 31$ & $2B$ &  4  \\
$HS$ & $2^9\cdot 3^2\cdot 5^3\cdot 7\cdot 11$ & $2C$, $2D$ &  5  \\
$He$ & $2^{10}\cdot 3^3\cdot 5^2\cdot 7^3\cdot 17$ & $2C$&  5  \\
$Suz$ & $2^{13}\cdot 3^7\cdot 5^2\cdot 7\cdot 11\cdot 13$ & $2C$, $2D$&  5  \\
$HN$ & $2^{14}\cdot 3^6\cdot 5^6\cdot 7\cdot 11\cdot 19$ & $2C$ &  5  \\
$Fi_{22}$ & $2^{17}\cdot 3^9\cdot 5^2\cdot 7\cdot 11\cdot 13$ & $2D$, $2E$, $2F$ &  7  \\
${Fi_{24}}'$ & $2^{21}\cdot 3^{16}\cdot 5^5\cdot 7^3\cdot 11\cdot 13\cdot 23\cdot 29$ & $2C$, $2D$ &  8  \\
\hline
\end{tabular}
\end{table}
\end{center}
}

The main result of this paper is as follows.

\begin{thm}\label{theo1}
Let $x\in\operatorname{Aut}(S)\setminus S$~be an automorphism of prime order of a sporadic group~$S$. Then $3\leqslant \alpha^{\phantom{S}}_{S}(x)\leqslant 4$, except when $(S,x)=({Fi_{24}}', 2C)$ and $\alpha^{\phantom{S}}_{S}(x)=5$.
\end{thm}

Combining this result with \cite[Theorem~3.1]{DMPZ} and Proposition~\ref{Prop} we will also prove the following assertion which includes the cases of inner and outer automorphisms of prime and composite order.

\begin{thm}\label{theo2}
Let $x\in\operatorname{Aut}(S)$~be a nonidentity automorphism of a sporadic group~$S$. Then $\alpha^{\phantom{S}}_{S}(x)\leqslant 4$, except in the following cases:
\begin{enumerate}
    \item[$1)$] $(S,x)=(Fi_{22}, 2A)$ and $\alpha^{\phantom{S}}_{S}(x)=6$;
    \item[$2)$] $(S,x)=(Fi_{23}, 2A)$ and $\alpha^{\phantom{S}}_{S}(x)=6$;
    \item[$3)$] $(S,x)=({Fi_{24}}', 2C)$ and $\alpha^{\phantom{S}}_{S}(x)=5$.
\end{enumerate}
\end{thm}

The following problem still remains open.
\begin{prob}
Find ~$\alpha^{\phantom{S}}_{S}(x)$ for every sporadic simple group~$S$ and its non\-identity automorphism~$x$.
\end{prob}

In order to solve this problem, we have to determine the precise value of $\alpha^{\phantom{S}}_{S}(x)$ in the above-mentioned cases \ref{item3})\,--\,\ref{item6}), where \cite[Theorem~3.1]{DMPZ} does not give such a  value, as well as determine
whether $\alpha^{\phantom{S}}_{S}(x)$ equals $3$ or $4$ for $S$ and $x$ from Table \ref{Spor}, except $(S,x)=({Fi_{24}}',2C)$, where we know that~$\alpha^{\phantom{S}}_{S}(x)=5$.

\section{Proof of Theorems~\ref{theo1} and~\ref{theo2} }

\begin{proof}[Proof of Theorem~\ref{theo1}] As we have already mentioned, the inequality   $3\leqslant \alpha^{\phantom{S}}_{S}(x)$ for an involution $x$ follows from the fact that every two involutions generate a solvable group \cite[Lemma~2.14]{Isaacs1}.

Taking account of Proposition~\ref{Prop} and the data in Table~\ref{Spor},
we will prove Theorem~\ref{theo1} once we establish that $\alpha^{\phantom{S}}_{S}(x)\leqslant 4$ when   $(S,x)$ is in the following list:
\begin{multline*} (HS,2C),(HS,2D), (He,2C), (Suz,2C),(Suz,2D),(HN,2C),\\ (Fi_{22},2D),(Fi_{22},2E),(Fi_{22},2F), ({Fi_{24}}',2D).
\end{multline*}

Since $O_2(G)=1$ for $G=\langle x,S\rangle$, the Baer--Suzuki theorem \cite[Theorem~2.12]{Isaacs1} implies that~$x$, together with some conjugate~$x^t$, generates a subgroup containing a nonidentity element~$y$ of odd order. Also,  $y\in S$, because $|G:S|=2$. If
$\alpha^{\phantom{S}}_{S}(y)=2$ then by definition ~$S$ contains an element $g\in G$ such that
$$S=\langle y, y^g\rangle\leqslant \langle x,x^t, x^g,x^{tg}\rangle.$$
This will imply the inequality ${\alpha^{\phantom{S}}_{S}(x)\leqslant 4}$ and the claim will follow.

In view of \cite[Theorem~3.1]{DMPZ}, we have $\alpha^{\phantom{S}}_{S}(y)=2$ for a nonidentity $y\in S$ of odd order, where $(S,x)$ is from the above list, unless
$$(S,y)\in\{(Suz,3A),(Fi_{22},3A),(Fi_{22},3B),({Fi_{24}}',3A),({Fi_{24}}',3B)\}.$$
In particular, $\alpha^{\phantom{S}}_{S}(y)=2$ if $S$~is one of $HS$, $He$, or $HN$.

In the remaining part of the proof, we are going to establish that even when $S$~is one of the groups $Suz$, $Fi_{22}$, or ${Fi_{24}}'$, the conjugate $x^t$ can be chosen so that the product $y=xx^t$ would be of odd order and belong to neither class $3A$ nor $3B$. Under this choice, we have $\alpha^{\phantom{S}}_{S}(y)=2$ as required.

We use the known fact from character theory that,
given elements $a,b$ and $c$ of a group $G$, the number $\mathrm{m}(a,b,c)$ of pairs
$(u,v)$, where $u$ is conjugate to $a$, $v$ is conjugate to $b$, and
$uv=c$, can be found from the character table using the formula
$$
\mathrm{m}(a,b,c)=\frac{|G|}{|\operatorname{C}_G(a)||\operatorname{C}_G(b)|}\sum\limits_{\chi\in\mathrm{Irr}(G)}\frac{\chi(a)\chi(b)\overline{\chi(c)}}{\chi(1)},
$$
see \cite[Exercise~(3.9), p.~45]{Isaacs}. To make sure that $x$ and some of its conjugates generate a subgroup containing a conjugate of~$y$, it is sufficient to show that $\mathrm{m}(x,x,y)>0$. We use the character tables of $Suz.2$, $Fi_{22}.2$, and ${Fi_{24}}'.2$ available in both \cite{atlas} and the computer algebra system \texttt{GAP} \cite{GAP}. The remaining cases can be treated using the \texttt{GAP} function

\medskip
\noindent
\texttt{> ClassMultiplicationCoefficient()}

\medskip
\noindent
to calculate $\mathrm{m}(x,x,y)$ for $x$ and $y$ listed in Table \ref{tab:cl_mult}. It turns out that in all these cases $y$ can be chosen from the class denoted by~$3C$ in~\cite{atlas}.
The proof of Theorem~\ref{theo1} is complete.\end{proof}

\begin{center}
\begin{table}\caption{Some class multiplication coefficients}
    \label{tab:cl_mult}
    \begin{tabular}{|c|l|}
        \hline
        $(S,x)$ \rule[-0.9ex]{0pt}{3.3ex} & $\mathrm{m}(x,x,y)$ \\
        \hline
        $(Suz,2C)$ & $\mathrm{m}(2C,2C,3C)=45$  \rule{0pt}{2.3ex} \\
        $(Suz,2D)$ & $\mathrm{m}(2D,2D,3C)=45$ \\
        $(Fi_{22},2D)$ & $\mathrm{m}(2D,2D,3C)=3$ \\
        $(Fi_{22},2E)$ & $\mathrm{m}(2E,2E,3C)=729$ \\
        $(Fi_{22},2F)$ & $\mathrm{m}(2F,2F,3C)=1080$ \\
        $({Fi_{24}}',2D)$ & $\mathrm{m}(2D,2D,3C)=1224720$ \rule[-0.9ex]{0pt}{0pt}\\
        \hline
    \end{tabular}
\end{table}
\end{center}

\bigskip
\begin{proof}[Proof of Theorem~\ref{theo2}]
For inner nonidentity automorphisms of $S$, the claim holds in view of~\cite[Theorem~3.1]{DMPZ} and Proposition~\ref{Prop}. For elements of prime order in ${\operatorname{Aut}(S)\setminus S}$, the claim holds by Theorem~\ref{theo1}. It remains to consider the elements of composite order in the difference ${\operatorname{Aut}(S)\setminus S}$. Let $x$~be such an element. If $y$ is a power of $x$, $|y|$ is prime, and $(S,y)\notin\{(Fi_{22},2A),({Fi_{24}}',2C)\}$ (for example, when $x$ is not a $2$-element), then $$\alpha^{\phantom{S}}_{S}(x)\leqslant\alpha^{\phantom{S}}_{S}(y)\leqslant 4,$$
and the claim holds for $x$. Consequently, we may assume that $S$~is either  $Fi_{22}$ or ${Fi_{24}}'$, $x$~is a $2$-element, and $y$ is the power of $x$ of order $2$ in class $2A$ if $S=Fi_{22}$ and in class $2C$ if $S={Fi_{24}}'$. Then either $x$ or some power of $x$ has order $4$. The character tables of $Fi_{22}.2$ and ${Fi_{24}}'.2$ from \cite{GAP} contain information about which  conjugacy classes the prime powers of elements belong to. This information implies that  both $Fi_{22}.2$ and ${Fi_{24}}'.2$ have no elements of order $4$ whose squares would belong to classes $2A$ or~$2C$. The proof is complete.
\end{proof}

The authors are grateful to Prof. Andrey S. Mamontov for helpful consultations.

\medskip

{\em Acknowledgement.}\/ The work of the first author was financially supported by the Russian Science Foundation (Grant 19-71-10017-$\Pi$). The work of the second author was financially supported by the State Contract of the Sobolev Institute of Mathematics  (Project FWNF-2022-0002).

\end{document}